\documentclass[11pt]{amsart}
\usepackage{amsmath, amsthm, amssymb,amscd}
\usepackage{graphicx}
\usepackage[all,cmtip]{xy}
\usepackage[english]{babel}
\newcommand{\proj}{\mathbb{P}}

\newcommand{\seq}{\subseteq}

\newtheorem{thm}{Theorem}[section]
\newtheorem*{thm-nl}{Theorem}
\newtheorem*{prop-nl}{Proposition}

\newtheorem{lem}[thm]{Lemma}

\def\OO{\mathcal{O}}

\def\cM{\mathcal{M}}

\def\Pic0{{\rm Pic}^0(X)}

\newtheorem{cor}[thm]{Corollary}
\newtheorem*{cor-nl}{Corollary}
\newtheorem{conjecture}[thm]{Conjecture}
\newtheorem*{conjecture-nl}{Conjecture}
\newtheorem*{quest-nl}{Question}
\newtheorem*{quests-nl}{Questions}

\newtheorem{prop}[thm]{Proposition}

\theoremstyle{remark}

\title{{The extremal Secant Conjecture for curves of arbitrary gonality}}

\author[M. Kemeny]{Michael Kemeny}

\address{Humboldt-Universit\"at zu Berlin, Institut f\"ur Mathematik,  Unter den Linden 6
\hfill \newline\texttt{}
 \indent 10099 Berlin, Germany} \email{{\tt michael.kemeny@gmail.com}}

\begin{document}
\maketitle
\begin{abstract}
We prove the Green--Lazarsfeld Secant Conjecture \cite[Conjecture (3.4)]{green-lazarsfeld-projective} for extremal line bundles on curves of arbitrary gonality, subject to explicit genericity assumptions.
\end{abstract}
\section{Introduction}
Consider a smooth projective curve $C$ of genus $g$ and $L$ a globally generated line bundle of degree $d$. We define the Koszul group $K_{i,j}(C,L)$ as the middle cohomology of 
$$\bigwedge^{i+1}H^0(L) \otimes H^0((j-1)L) \to \bigwedge^{i}H^0(L) \otimes H^0(jL) 
\to \bigwedge^{i-1}H^0(L) \otimes H^0((j+1)L)$$
As is well known, the Koszul groups give the same data as the modules appearing in the minimal free resolution of the $\text{Sym}\;H^0(C,L)$ module $\bigoplus_q H^0(C,qL)$. In the case where $L$ is very ample and the associated embedding is projectively normal, $\bigoplus_q H^0(C,qL)$ is just the homogeneous coordinate ring of the embedded curve $\phi_L: C \hookrightarrow \proj^r$.\\

The pair $(C,L)$ is said to satisfy \emph{property $(N_p)$} if we have the vanishings $$K_{i,j}(C,L)=0 \text{ for} \; i \leq p, \; \; j \geq 2.$$ Then $\phi_L: C \hookrightarrow \proj^r$ is projectively normal if and only if $(C,L)$ satisfies $(N_0)$, whereas the ideal of $C$ is generated by quadrics if, in addition, it satisfies $(N_1)$.\\

A beautiful conjecture of Green--Lazarsfeld gives a necessary and sufficient criterion for $(C,L)$ to satisfy $(N_p)$. To state the conjecture, a line bundle $L$ is called \emph{$p$-very ample} if and only if for every effective divisor $D$ of degree $p+1$ the evaluation map $$ev: H^0(C,L) \to H^0(D,L_{|_D})$$ is surjective. Equivalently, $L$ is \emph{not} $p$-very ample if and only if $C \seq \proj^r$ admits a $(p+1)$-secant $p-1$-plane. We then may state \cite{green-lazarsfeld-projective}: 
\begin{conjecture}[G-L Secant Conjecture] \label{secant-conj}
Let $L$ be a globally generated line bundle of degree $d$ on a curve $C$ of genus $g$ such that
$$d \geq 2g+p+1-2h^1(C,L)-\text{Cliff}(C).$$
Then $(C,L)$ fails property $(N_p)$ if and only if $L$ is not $p+1$-very ample.
\end{conjecture}
It is rather straightforward to see that if $L$ is not $p+1$ very ample, or, equivalently, $L$ admits a $(p+2)$-secant $p$-plane, then $K_{p,2}(C,L)$ is nonzero. The difficulty in establishing the above conjecture is thus to go in the other direction, that is, to construct a secant plane out of a syzygy in $K_{p,2}(C,L)$.\\

In the case $H^1(C,L) \neq 0$, it is well-known that the Secant Conjecture reduces to the Green's Conjecture, which holds for the generic curve in each gonality stratum, \cite{voisin-even}, \cite{voisin-odd}. Thus we will henceforth assume $H^1(C,L)=0$. If $d \geq 2g+p+1$, then $L$ is automatically $p+1$ very ample and further $L$ satisfies property $(N_p)$ by \cite[Thm.\ 4.a.1]{green-koszul}. In particular, we may assume both $\text{Cliff}(C) \geq 1$ and $d \leq 2g+p$. In this case, the line bundle $L$ of degree $d$ fails to be $p+1$ very ample if and only if 
$$L-K_C \in C_{p+2}-C_{2g-d+p},$$
where $C_i \seq \text{Pic}^{i}(C)$ is the image of the $i$-th symmetric product of $C$ under the Abel--Jacobi map (we set $C_0:= \emptyset$).

\medskip
In a joint work with Gavril Farkas, we established the Secant Conjecture for general line bundles on general curves. Moreover, under certain assumptions on the degree, we were able to prove effective versions of the Secant Conjecture. One of our main results was a proof of the conjecture for odd genus curves of maximal Clifford index and line bundles of degree $d=2g$; this is the so-called `divisorial case' of the conjecture. To be precise, we showed:
\begin{thm}[\cite{generic-secant}] \label{general-thm}
Let $C$ be a smooth curve of odd genus $g$ and with a line bundle $L\in \mathrm{Pic}^{2g}(C)$. Then one has the equivalence
$$K_{\frac{g-3}{2},2}(C,L)\neq 0  \ \Leftrightarrow \ \mathrm{Cliff}(C)<\frac{g-1}{2} \  \  \mathrm{or} \  \  \ L-K_C \in C_{\frac{g+1}{2}}-C_{\frac{g-3}{2}}.$$
\end{thm}
The latter condition $ L-K_C \in C_{\frac{g+1}{2}}-C_{\frac{g-3}{2}}$ is equivalent to $L$ failing to be $\frac{g-1}{2}$-very ample.\\

In the case where $C$ is Brill--Noether--Petri general of even genus $g$, we have a similar statement:
\begin{thm}[\cite{generic-secant}]  \label{even}
The Green-Lazarsfeld Conjecture holds for a Brill--Noether--Petri general curve $C$ of even genus and every line bundle $L\in \mathrm{Pic}^{2g+1}(C)$, that is,
$$K_{\frac{g}{2}-1,2}(C,L)\neq 0 \ \Leftrightarrow \ L-K_C\in C_{\frac{g}{2}+1}-C_{\frac{g}{2}-2}.$$
\end{thm}
\medskip

The main result of this paper is the an analogue of Theorem \ref{general-thm} in the case of curves of arbitrary gonality, satisfying the \emph{linear growth condition} of Aprodu, \cite{aprodu-remarks}. In this case, $p$ takes on the extremal value $p=g-k$:
\begin{thm}\label{gonality}
Let $C$ be a smooth curve of genus $g$ and gonality $3 \leq k<\lfloor \frac{g}{2} \rfloor+2$. Assume $C$ satisfies the following linear growth condition
$$\dim W^1_{k+n}(C) \leq n \ \ \text{for all} \; \; 0 \leq n \leq g-2k+2 .$$
Then the G-L Secant Conjecture holds for every line bundle $L\in \mathrm{Pic}^{3g-2k+3}(C)$, that is, one has the equivalence
$$K_{g-k,2}(C,L)\neq 0  \ \Leftrightarrow  \  \ L-K_C \in C_{g-k+2}-C_{k-3}.$$
\end{thm}
The proof is by reducing to the case of Theorem \ref{general-thm}, using arguments similar to those in \cite[\S 6]{generic-secant} and \cite{aprodu-remarks}. Note that $h^1(L)=0$ is automatic for $L\in \mathrm{Pic}^{3g-2k+3}(C)$ as above. The condition $L-K_C \in C_{g-k+2}-C_{k-3}$ is equivalent to the statement that $L$ fails to be $g-k+1$ very ample. Note that for the value $p=g-k$ the set of $L$ which fail to be $p+1$ very ample defines a divisor in the Jacobian; thus this case is of particular interest.

From the main theorem we easily deduce the following statement giving an effective criteria for the vanishing $K_{p,2}(C,L)=0$ for nonspecial line bundles in the case where the inequality in the Secant Conjecture is an equality.
\begin{thm} \label{extremal}
Let $C$ be a smooth curve of genus $g$ and gonality $3 \leq k<\lfloor \frac{g}{2} \rfloor+2$. Assume $C$ satisfies the following linear growth condition
$$\dim W^1_{k+n}(C) \leq n \ \ \text{for all} \; \; 0 \leq n \leq g-2k+2 .$$
Let $L \in \text{Pic}^{2g+p-k+3}(C)$ be nonspecial. If $p>g-k$, then $K_{p,2}(C,L) \neq 0$. On the other hand, assume $p \leq g-k$ and, in addition, we have the two conditions
\begin{align}
&H^1(C,2K_C-L)=0, \\
&\text{The Secant Variety $V^{g-p-4}_{g-p-3}(2K_C-L)$ has expected dimension $g-k-p-1$}.
\end{align}
Then $K_{p,2}(C,L)=0$.
\end{thm}
Notice that, if the condition $H^1(C,2K_C-L) \neq 0$, then $L-K_C$ is effective, which obviously implies that $L$ is not $p+1$-very ample. In this case, we already know $K_{p,2}(C,L) \neq 0$, from the easy direction of the G-L Secant conjecture. So the ``interesting'' assumption is really the second one\footnote{In Theorem 1.7 of \cite{generic-secant}, we forgot to explicitly state the assumption $L-K_C$ not effective. When $L-K_C$ is effective, then the expected dimension of $V^{g-p-4}_{g-p-3}(2K_C-L)$ is strictly less than $g-k-p-1$, so this assumption was actually implicit in Theorem 1.7. As explained above, the case $L-K_C$ effective is of no interest, as then $L$ trivially fails to be $p+1$ very ample.}.

In the case $p \leq g-k$, both the conditions of Theorem \ref{extremal} hold for a general line bundle $L \in \text{Pic}^{2g+p-k+3}(C)$. In particular we get, when combining with the results of \cite{generic-secant}:
\begin{cor} \label{cor}
Let $C$ be a general curve of genus $g$ and gonality $k \geq 3$ and let $L \in \text{Pic}^{2g+p-k+3}(C)$ be a general, nonspecial, line bundle. Then the Green--Lazarsfeld Secant conjecture holds for $(C,L)$, i.e.\ $$K_{p,2}(C,L) \neq 0  \ \Leftrightarrow  \  \ Êp>g-k.$$ 
\end{cor}
\medskip

\noindent {\bf Acknowledgments:} We thank Holger Brenner for an interesting discussion in Osnabr\"uck which lead to this work. I thank Gavril Farkas for sharing many ideas on these topics with me. This work was supported by the DFG Priority Program 1489 \emph{Algorithmische Methoden in Algebra, Geometrie und Zahlentheorie}.

\section{Proof of the Theorem}
Let $C$ be a smooth curve of genus $g$ and gonality $3 \leq k<\lfloor \frac{g}{2} \rfloor+2$; this covers all cases other than $C$ hyperelliptic or $g$ odd and $C$ of maximal gonality.
Assume in addition $C$ satisfies the linear growth condition:
$$\dim W^1_{k+n}(C) \leq n \ \ \text{for all} \; \; 0 \leq n \leq g-2k+2 .$$
Pick $g-2k+3$ \emph{general} pairs of points $(x_i,y_i)$. Let $D$ be the semistable curve obtained by adding $g-2k+3$ smooth, rational components $R_i$ to $C$, each one of which meeting $C$ at the pair $(x_i,y_i)$. The curve $D$ is illustrated in the figure below. It has arithmetic genus $2g-2k+3$.

Let $L$ be a very ample line bundle on $C$ of degree $3g-2k+3$. Write
$$ L= \mathcal{O}_{C}(z_1+\ldots+z_{3g-2k+3})$$
for distinct points $z_1, \ldots, z_{3g-2k+3}$ which avoid all $(x_i,y_i)$. For each $1 \leq i \leq g-2k+3$, choose points $w_i \in R_i$ distinct from $x_i,y_i$. Let $T$ denote the union of the points $z_j$ and $w_i$. Since $T$ avoids all nodes it makes sense to set $N:=\mathcal{O}_D(T)$. Notice that $N$ defines a \emph{balanced} line bundle on the quasi-stable curve $D$, and that $D$ is $4g-4k+6$-general, in the sense of \cite{cap-neron}. In particular, $N$ defines a (stable) point in Caporaso's compactified Jacobian $\displaystyle \overline{P}^{4g-4k+6}(X)$, where $X$ is the stabilisation of $D$, i.e.\ the nodal curve obtained from $C$ by identifying $x_i$ with $y_i$. \\

\begin{figure}
\centering
\includegraphics[scale=0.33]{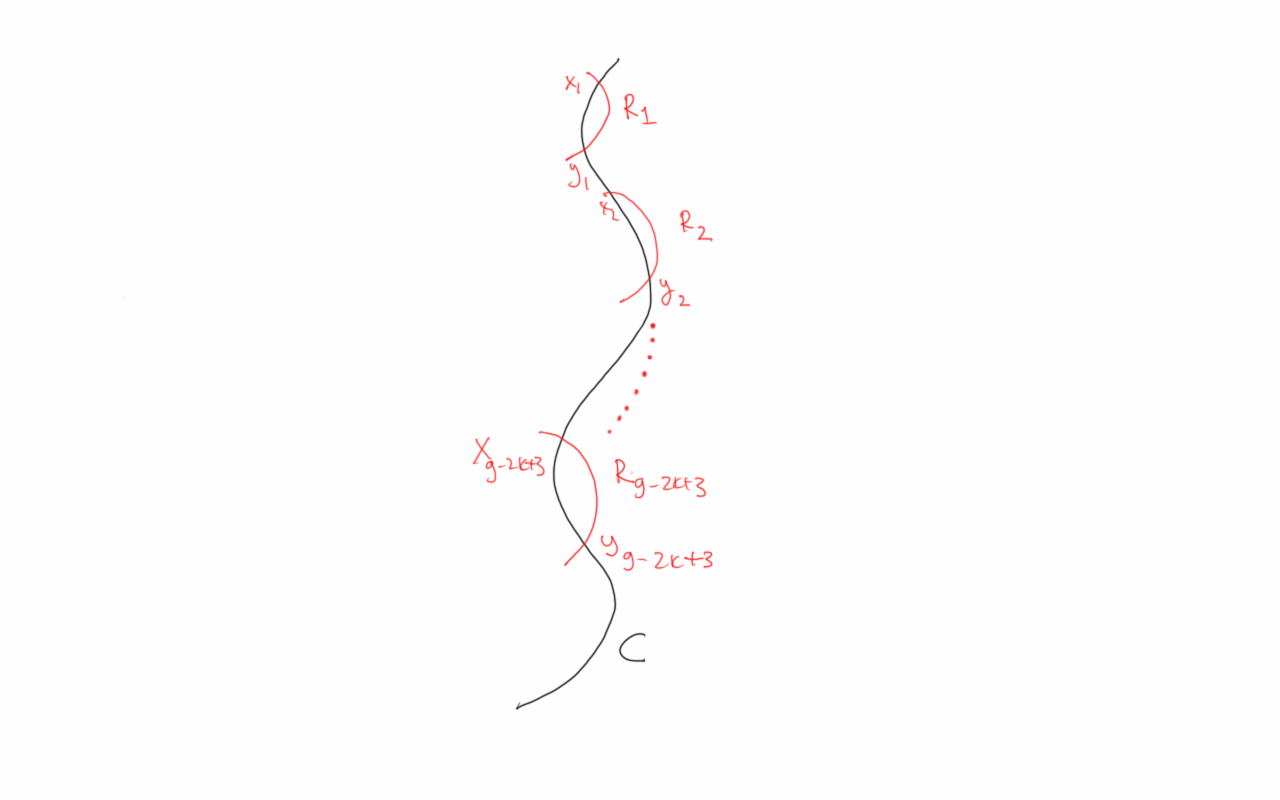}
\caption{The curve $D$}
\end{figure}

The curve $D$ together with the marking $\{z_1, \ldots, z_{3g-2k+3},w_1, \ldots, w_{g-2k+3}\}$ defines a point $[D] \in \overline{\mathcal{M}}_{2g-2k+3, 2(2g-2k+3)}$. Let $\overline{\mathcal{M}}^{va}_{2g-2k+3, 2(2g-2k+3)}$ denote the open locus  of marked stable curves $D'$ such that the marking defines a \emph{very-ample} line bundle $N'$ with $H^1(D',N')=0$, and set
$\mathcal{M}^{va}_{2g-2k+3, 2(2g-2k+3)}:=\overline{\mathcal{M}}^{va}_{2g-2k+3, 2(2g-2k+3)} \cap \mathcal{M}_{2g-2k+3, 2(2g-2k+3)}$. In \cite[Thm.\ 1.6]{generic-secant}, we established the following equality of closed sets
$$\overline{\mathfrak{Syz}}=\overline{\mathfrak{Sec}}\cup \overline{\mathfrak{Hur}}.$$
Here $\overline{\mathfrak{Syz}}$ denotes the closure of the locus $\mathfrak{Syz}$ of smooth, marked curves $\mathfrak{Syz}$ such that the marking defines a very ample line bundle with a certain unexpected syzygy
$$\mathfrak{Syz}:=\Bigl\{[B,x_1, \ldots, x_{2(2g-2k+3)}]\in \cM^{va}_{2g-2k+3,2(2g-2k+3)}: K_{g-k,2}\bigl(B, \OO_B(x_1+\cdots+x_{2(2g-2k+3)})\bigr)\neq 0\Bigr \},$$
whereas $\overline{\mathfrak{Sec}}$ denotes the closure of the locus of smooth, marked curves $\mathfrak{Syz}$ such that the marking defines a line bundle which fails to be $g-k+1$-very ample
$$\mathfrak{Sec}:=\Bigl\{[B,x_1, \ldots, x_{2(2g-2k+3)}]\in \cM_{2g-2k+3,2(2g-2k+3)}: \OO_B(x_1+\cdots+x_{2(2g-2k+3)})\in K_B+B_{g-k+2}-B_{g-k}\Bigr \}.$$
and $\overline{\mathfrak{Hur}}$ is the closure of the \emph{Hurwitz} divisor of curves which are $g-k+2$ gonal.
The following result is due to Aprodu.
\begin{prop}[\cite{aprodu-remarks}] \label{marians-prop}
The marked curve $[D]$ lies outside $\overline{\mathfrak{Hur}} \seq \overline{\mathcal{M}}_{2g-2k+3, 2(2g-2k+3)}$.
\end{prop}
\begin{proof}
Let $X$ be the stabilisation of $D$ as above. By \cite{hir-ram}, it suffices to show $$K_{g-k+1,1}(X,\omega_X)=0,$$
see also \cite[Proposition\ 7]{aprodu-remarks}. This is implied by the linear growth assumption on $C$ and the generality of the points $(x_i, y_i)$, see the proof of \cite[Thm.\ 2]{aprodu-remarks}.
\end{proof}
The following lemma is similar to \cite[Proposition 7]{aprodu-remarks}.
\begin{lem} \label{near-obvious}
Assume $[D] \in \overline{\mathcal{M}}_{2g-2k+3, 2(2g-2k+3)}$ lies outside $$\overline{\mathfrak{Syz}} \seq \overline{\mathcal{M}}_{2g-2k+3, 2(2g-2k+3)}.$$
Then $K_{g-k,2}(D,N)=0$.
\end{lem}
\begin{proof}
The only reason this lemma is not totally obvious is that $\mathfrak{Syz}$ was defined as the closure of \emph{smooth}, marked curves with extra-syzygies. However, the determinantal description from \cite[\S 6]{generic-secant} can be extended verbatim to the open locus $\overline{\mathcal{M}}^{va}_{2g-2k+3, 2(2g-2k+3)}$ of marked stable curves $D'$ such that the marking defines a very-ample line bundle $N'$ with $H^1(D', N')=0$; see also \cite[\S 2]{farkas-syzygies}, \cite{farkas-koszul}. Indeed, the only thing which needs checking is that we continue to have $H^1(D', \bigwedge^{g-k}M_{N'} \otimes N'^{2})$. This follows from the short exact sequence
$$0 \to \bigwedge^{g-k+1}M_{N'} \otimes N' \to \bigwedge^{g-k+1}H^0(N') \otimes N' \to  \bigwedge^{g-k}M_{N'} \otimes N'^{2} \to 0,$$
and the assumption $H^1(D', N')=0$.

Thus we get a \emph{divisor} $\mathfrak{Syz}^{va} \seq \overline{\mathcal{M}}^{va}_{2g-2k+3, 2(2g-2k+3)}$ which coincides with $\mathfrak{Syz}$ on $\mathcal{M}^{va}_{2g-2k+3, 2(2g-2k+3)}$. Now, the fact that $L$ is very ample implies that $N$ is very ample; indeed, $H^0(C,L) \simeq H^0(D,N)$, and $\phi_N: D \to \proj^{g-2k+3}$ embeds $D$ as the union of the curve $C$ (embedded by $L$) together with $g-2k+3$ secant lines $R_i$. So $[D] \in \overline{\mathcal{M}}^{va}_{2g-2k+3, 2(2g-2k+3)}$. Riemann--Roch now gives $H^1(D,N')=0$. The point $[D]$ lies on precisely one boundary component of $\overline{\mathcal{M}}_{2g-2k+3, 2(2g-2k+3)}$, namely the component $\delta_{irr}$ whose general point is an integral curve with one node; see \cite{arb-corn-calculating} for details on the boundary of $\overline{\mathcal{M}}_{g,n}$.

Thus, it suffices to show that $\mathfrak{Syz}^{va}$ does not contain $\delta_{irr}$. This follows easily from \cite[Thm.\ 1.8]{generic-secant}. Indeed, it suffices to show that there exist integral, singular curves with nodal singularities in the linear system $|L|$ on the K3 surface $Z_{2g-2k+3}$ from Section 3 of \emph{loc.\ cit.\ }\footnote{There is a typo in the statement of \cite[Thm.\ 1.8]{generic-secant}, namely we should have $(C)^2=4i$. This typo is not repeated in \cite[\S 3]{generic-secant}.} For this, one can degenerate to the hyperelliptic K3 surface $\hat{Z}_{2g-2k+3}$ as in \cite[\S 3]{generic-secant}, and take a general curve $A$ in the base point free linear system $|L-E|$ which meets a given elliptic curve $B \in |E|$ transversally. The nodal curve $A+B$ then deforms to an integral nodal curve in $|L|$.

\end{proof}
We next compare difference varieties with secant varieties, see \cite[VIII.4]{ACGH}, \cite[\S 2]{generic-secant} for background.
\begin{lem} \label{lem1}
For any $0 \leq j \leq g-2k+3$, the inclusion
$$L-C_{2(g-2k+3-j)}-K_C \seq C_{k-1+j}-C_{g-k-j}, $$
of closed subvarieties of $\text{Pic}^{2k+2j-g-1}(C) $ implies that the following dimension estimate holds
$$ \dim V^{2g-3k-j+4}_{2g-3k-j+5}(L) \geq 2g-2j-4k+6.$$
\end{lem} 
\medskip
Note that the expected dimension of the secant variety $V^{2g-3k-j+4}_{2g-3k-j+5}(L)$ is $2g-2j-4k+5$, so the inclusion above implies that the secant variety has dimension higher than expected.
\begin{proof}
From the inclusion $L-C_{2(g-2k+3-j)}-K_C \seq C_{k-1+j}-C_{g-k-j} $, we have that, for every effective divisor $D$ of degree $2(g-2k+3-j)$, there exists an effective divisor $E$ of degree $k-1+j$ such that
$$ [L-(D+E)] \in K_C-C_{g-k-j}.$$
As $h^0(C,K_C)=g$, this implies that $L-(D+E)$ has at least $k+j$ sections. This is equivalent to $$[D+E] \in  V^{2g-3k-j+4}_{2g-3k-j+5}(L).$$ Let $C^{(i)}$ denote the $i$-th symmetric product of $C$. There are only finitely many possible $D' \in C^{(2(g-2k+3-j))}$ such that we have the equality of divisors $$[D+E]=[D'+E'] \in C^{(2g-3k-j+5)}$$ for some effective divisor $E'$ of degree $k-1+j$. Hence the dimension of $V^{2g-3k-j+4}_{2g-3k-j+5}(L)$ is at least $2(g-2k+3-j)$.
\end{proof}
We now apply \cite[Remark 4.2]{aprodu-sernesi} to show that if $L$ as above is $g-k+1$ very ample, then none of the secant loci from the previous lemma can have excess dimension.  
\begin{lem} \label{lem2}
Assume $L$ as above is $g-k+1$ very ample. Then $$\dim V^{2g-3k-j+4}_{2g-3k-j+5}(L) = 2g-2j-4k+5$$ for all $0 \leq j < g-2k+3$, whereas $V^{g-k+1}_{g-k+2}(L)=\emptyset$.
\end{lem}
\begin{proof}
Firstly note that, if $0 \leq j < g-2k+3$, then all secant loci $V^{2g-3k-j+4}_{2g-3k-j+5}(L)$ are nonempty by \cite[pg.\ 356]{ACGH}. For $j=g-2k+3$ the secant locus $V^{g-k+1}_{g-k+2}(L)=\emptyset$ by the assumption that $L$ is $g-k+1$ very ample. Suppose there exists $0 \leq j < g-2k+3$ with $\dim V^{2g-3k-j+4}_{2g-3k-j+5}(L) > 2g-2j-4k+5$. By \cite[Remark 4.2]{aprodu-sernesi}, $\dim V^{g-k+2}_{g-k+3}(L)\geq 2$. 

Consider the Abel--Jacobi map $\pi \; : V^{g-k+2}_{g-k+3}(L) \to \text{Pic}^{g-k+3}(C)$. We claim that $\pi$ is finite. Indeed, otherwise we would have a one-dimensional family of 
$[D_t] \in V^{g-k+2}_{g-k+3}(L)$ with $\mathcal{O}_{C}(D_t)$ constant. Then the line bundle $K_C-L+D_t$ is independent of $t$, and furthermore it is effective, since $[D_t] \in V^{g-k+2}_{g-k+3}(L)$. Let $Z \in |K_C-L+D_t|$ and $s \in \text{Supp}(Z)$; the assumption $k \geq 3$ ensures $\deg(K_C-L+D_t) \geq 1$. There exists some $t'$ such that $s \in \text{Supp}(D_{t'})$ and let $D':=D_{t'}-s$. Then $Z-s \in |K_C-L+D'|$, so $K_C-L+D'$ is effective, and $D' \in V^{g-k+1}_{g-k+2}(L)$, contradicting that $V^{g-k+1}_{g-k+2}(L)=\emptyset$. Thus we have that $\pi$ is finite. 

We will now apply \cite{FHL} to see that $V^{g-k+1}_{g-k+2}(L)\neq \emptyset$ (cf.\ \cite[Remark 4.4]{aprodu-sernesi} and the proof of \cite[Thm.\ 1.5]{generic-secant}). This contradiction will finish the proof. Indeed, for any point $p \in C$, we can find an irreducible, closed curve $S \seq V^{g-k+2}_{g-k+3}(L)$ such that, for all $s \in S$, the corresponding divisor $[D_s] \in V^{g-k+2}_{g-k+3}(L)$ passes through $p$, so that $D'_s:=D_s-p$ is an effective divisor. Now consider the Abel--Jacobi map
\begin{align*}
p \; : V^{g-k+2}_{g-k+3}(L) &\to \text{Pic}^{k-2}(C) \\
[D] & \mapsto K_C-L+D
\end{align*}
This image $p(S)$ is a closed curve, each point of which parametrises an effective line bundle. By \cite{FHL}, there exists an $s \in S$ with $K_C-L+D_s-p=K_C-L+D'_s$ effective. But this is the same as saying $D'_s \in V^{g-k+1}_{g-k+2}(L)$.
\end{proof}
We now record a lemma which we will shall need for the proof of the main theorem. 
\begin{lem} \label{good-lem}
Let $N$ be the balanced line bundle of degree $4g-4k+6$ as above and assume $$K_{g-k,2}(D,N)=0.$$ Then 
$$K_{g-k,2}(C, L)=0.$$
\end{lem}
\begin{proof}
Assume $K_{g-k,2}(C,L) \neq 0.$ Then $K_{g-k+2,0}(C,\omega_C;L) \neq 0$ by Koszul duality, \cite{green-koszul}. Likewise $K_{g-k,2}(D,N)=0$
if and only if $K_{g-k+2,0}(D,\omega_D;N)$. Note that $H^0(D,N) \simeq H^0(C,L)$, and the proof of Koszul duality using kernel bundles goes through unchanged in our case, even though $D$ is nodal, see \cite[Thm.\ 2.24]{aprodu-nagel}.
Restriction induces natural inclusions
\begin{align*}
 H^0(D,N) &\hookrightarrow H^0(C,L)\\
 H^0(D,\omega_D) &\hookrightarrow H^0(C,\omega_C(\sum_{i=1}^{g-2k+3} x_i+y_i ))\\
 H^0(D,N \otimes \omega_D) & \hookrightarrow H^0(C,L \otimes \omega_C(\sum_{i=1}^{g-2k+3} x_i+y_i ))
 \end{align*}
We thus get the following commutative diagram, where both vertical arrows are injective:
$$
\begin{CD}
{\bigwedge^{g-k+2} H^0(D,N)\otimes H^0(D,\omega_D)} @>{d_{g-k+2,0}}>>{\bigwedge^{g-k+1} H^0(D,N)\otimes H^0(D, N\otimes \omega_D) } \\
@V{}VV @V{}VV \\
{\bigwedge^{g-k+2} H^0(C,L)\otimes H^0(C,\omega_C(\sum_{i} x_i+y_i ))} @>{\tilde{d}_{g-k+2,0}}>> \bigwedge^{g-k+1} H^0(C,L)\otimes H^0(C, L\otimes \omega_C(\sum_{i} x_i+y_i )). \\
\end{CD}
$$
We have an isomorphism $H^0(D,N) \simeq H^0(C,L)$, and $H^1(C,L)=0$, so Riemann--Roch implies $H^1(D,N)=0$. 
\medskip

The image of the restriction map $H^0(D,\omega_D)\hookrightarrow H^0(C,\omega_C(\sum_{i} x_i+y_i ))$ includes $H^0(C,\omega_C) \seq H^0(C,\omega_C(\sum_{i} x_i+y_i ))$. We have a natural commutative diagram, where the vertical arrows are injective:
$$\begin{CD}
{\bigwedge^{g-k+2} H^0(C,L)\otimes H^0(C,\omega_C)} @>{d'_{g-k+2,0}}>>{\bigwedge^{g-k+1} H^0(C,L)\otimes H^0(C, L\otimes \omega_C) } \\
@V{}VV @V{}VV \\
{\bigwedge^{g-k+2} H^0(C,L)\otimes H^0(C,\omega_C(\sum_{i} x_i+y_i ))} @>{\tilde{d}_{g-k+2,0}}>> \bigwedge^{g-k+1} H^0(C,L)\otimes H^0(C, L\otimes \omega_C(\sum_{i} x_i+y_i )). \\
\end{CD}$$
Thus if $K_{g-k+2,0}(C,\omega_C;L) \neq 0$, then there exists a nonzero element of $\text{Ker}(\tilde{d}_{g-k+2,0})$ which lies in the image of $$\bigwedge^{g-k+2} H^0(D,N)\otimes H^0(D,\omega_D) \to \bigwedge^{g-k+2} H^0(C,L)\otimes H^0(C,\omega_C(\sum_{i} x_i+y_i ))$$ and thus $d_{g-k+2,0}$ is non-injective, so $K_{g-k+2,0}(D,\omega_D;N)\neq 0$.
\medskip
\end{proof}
We are now in a position to prove the main theorem.
\begin{proof}[Proof of Theorem \ref{gonality}]
Assume $$L-K_C \notin C_{g-k+2}-C_{k-3}.$$ We need to show $K_{g-k,2}(C,L)=0$.
From Lemma \ref{good-lem}, it suffices to prove $K_{g-k,2}(D,N)=0$. From Lemma \ref{near-obvious} and Proposition \ref{marians-prop}, 
it suffices to show that the marked curve $[D] \in \overline{\mathcal{M}}_{2g-2k+3, 2(2g-2k+3)}$ lies outside $\overline{\mathfrak{Sec}}$. For this it is sufficient to show
$$H^0(D,\bigwedge^{g-k}M_{K_D}(2K_D-N))=0,$$
by \cite[Prop.\ 3.6]{farkas-popa-mustata}. Here $M_{K_D}$ is the kernel bundle, defined by the exact sequence
$$ 0 \to M_{K_D} \to H^0(D,K_D) \otimes \mathcal{O}_D \to K_D \to 0.$$ Equivalently, if $\phi_{K_D} \; : D \to \proj^{2g-2k+2}$ is the canonical morphism, then $M_{K_D}\simeq \phi_{K_D}^* \Omega_{\proj^{2g-2k+2}}(1)$. Note that $\phi_{K_D}$ is not an embedding; indeed each component $R_i$ is contracted to a point.

We define subcurves of $D$ as such: for $1\leq k < g-2k+3$ let $$D_k:=C \cup R_{k+1} \cup \ldots R_{g-2k+3} $$ and set $D_{g-2k+3}=C$. Define $N_i:=N_{|_{D_i}}$. The Mayer--Vietoris sequence gives
\begin{align*}
0 &\to \bigwedge^{g-k} M_{K_D} \otimes (2K_D-N) \to \bigwedge^{g-k} M_{K_{D_1(x_1+y_1)}}(2K_{D_1}-N_1+2x_1+2y_1) \oplus \\
 &\mathcal{O}_{R_1}(-1)^{2g-2k+2 \choose g-k} \to \bigwedge^{g-k} M_{K_D} \otimes (2K_D-N)_{|_{x_1,y_1}} \to 0,
\end{align*}
using that $M_{{K_D} _{|_{D_1}}} \simeq M_{K_{D_1(x_1+y_1)}}$ (note that restriction induces an isomorphism $H^0(D,K_D) \simeq H^0(D_1,K_{D_1}(x_1+y_1))$).
\medskip

 So it suffices to show that the evaluation map 
$$H^0(D_1, \bigwedge^{g-k} M_{K_{D_1(x_1+y_1)}}(2K_{D_1}-N_1+2x_1+2y_1) ) \to H^0(D, \bigwedge^{g-k} M_{K_{D_1(x_1+y_1)}}(2K_{D_1}-N_1+2x_1+2y_1)_{|_{x_1,y_1}})$$
is injective, or $$H^0(D_1, \bigwedge^{g-k} M_{K_{D_1(x_1+y_1)}}(2K_{D_1}-N_1+x_1+y_1) )=0.$$
We have a short exact sequence
$$0 \to \bigwedge^{g-k}M_{K_{D_1}}\to \bigwedge^{g-k}M_{K_{D_1(x_1+y_1)}} \to \bigwedge^{g-k-1}M_{K_{D_1}}(-x_1-y_1) \to 0 ,$$
see for instance \cite[Remark p.\ 345]{beauville-stable}.
So, it is enough to show the following two vanishings
\begin{align*}
H^0(D_1, \bigwedge^{g-k} M_{K_{D_1}}(2K_{D_1}-N_1+x_1+y_1) )&=0\\
H^0(D_1, \bigwedge^{g-k-1} M_{K_{D_1}}(2K_{D_1}-N_1) )&=0
\end{align*}
For any semistable curve $Y$ and vector bundle $E$ on $Y$ we define $\Theta_E$ as the set of line bundles $M$ with
$$H^0(Y,E \otimes M) \neq 0.$$ We define $C^{sm,i}$ as the intersection of $C$ with the smooth locus of $D_i$.
As $x_1,y_1$ are general, it is enough to satisfy the following two conditions
\begin{align}
2K_{D_1}-N_1+C^{sm,1}_2 &\nsubseteq \Theta_{\bigwedge^{g-k} M_{K_{D_1}}} \\
2K_{D_1}-N_1&\nsubseteq \Theta_{\bigwedge^{g-k-1} M_{K_{D_1}}}
\end{align}
where the notation $C^{sm,i}_d$ refers to the $d$-th symmetric product of $C^{sm,i}$.

Using the Mayer--Vietoris sequence
$$0 \to \mathcal{O}_{D_{i-1}} \to \mathcal{O}_{D_{i}} \oplus \mathcal{O}_{R_{i}} \to \mathcal{O}_{x_{i},y_{i}} \to 0 ,$$
together with
$$0 \to \bigwedge^{p}M_{K_{D_i}}\to \bigwedge^{p}M_{K_{D_i(x_i+y_i)}} \to \bigwedge^{p-1}M_{K_{D_i}}(-x_i-y_i) \to 0 ,$$
and the generality of $x_1,y_1, \ldots, x_{2(2g-2k+3)},y_{2(2g-2k+3)}$ we see that, in order to verify the conditions 
\begin{align}
2K_{D_i}-N_i+C^{sm,i}_{2(i-j)} \nsubseteq \Theta_{\bigwedge^{g-k-j} M_{K_{D_i}}}, \; 0 \leq j \leq  i,
\end{align} 
it is enough to verify 
\begin{align}
2K_{D_{i+1}}-N_{i+1}+C^{sm,i+1}_{2(i+1-j)} \nsubseteq \Theta_{\bigwedge^{g-k-j} M_{K_{D_{i+1}}}}, \; 0 \leq j \leq  i+1.
\end{align} 
Hence it is enough to verify 
\begin{align}
2K_{C}-L+C_{2(g-2k+3-j)} \nsubseteq \Theta_{\bigwedge^{g-k-j} M_{K_{C}}}, \; 0 \leq j \leq  g-2k+3,
\end{align} 
or
\begin{align*}
L-C_{2(g-2k+3-j)}-K_C \nsubseteq C_{k+j-1}-C_{g-k-j}, \; 0 \leq j \leq  g-2k+3,
\end{align*} 
by Serre duality and \cite[Prop.\ 3.6]{farkas-popa-mustata}. This follows from Lemma \ref{lem1} and Lemma \ref{lem2}.
\end{proof}
Theorem \ref{extremal} now follows easily from Theorem \ref{gonality}.
\begin{proof}[Proof of Theorem \ref{extremal}]
In the case $p>g-k$, then each line bundle $L \in \text{Pic}^{2g+p-k+3}(C)$ fails to be $p+1$ very ample, \cite{generic-secant}. Thus, by the known direction of the Secant conjecture,
$K_{p,2}(C,L) \neq 0$. So assume $p \leq g-k$, $H^1(C,2K_C-L)=0$, $H^1(C,L)=0$ and that $V_{g-p-3}^{g-p-4}(2K_C-L)$ has the expected dimension $g-k-p-1$. Note that this implies that $L$ is $p+1$ very ample. Indeed, otherwise $$ L-K_C \in C_{p+2}-C_{k-3}$$ which implies that
$$2K_C-L-C_{g-k-p} \subseteq K_C+C_{k-3}-C_{g-k+2}, $$ which gives $\dim V_{g-p-3}^{g-p-4}(2K_C-L) \geq g-k-p,$ using the assumption $H^1(C,2K_C-L)=0$. In fact, this last inclusion is \emph{equivalent} to $\dim V_{g-p-3}^{g-p-4}(2K_C-L) \geq g-k-p$; use that a $1$-d family of divisors must pass through any given point.
In particular, the previous discussion shows $L$ is base point free. For a general, effective divisor $D$ of degree $g-k-p$, the argument above gives $L(D)-K_C \notin C_{g-k+2}-C_{k-3}$. By Theorem \ref{gonality}, we have 
$$K_{g-k,2}(C,L(D))=0. $$ By \cite[Prop.\ 2.1]{generic-secant} this implies $K_{p,2}(C,L)=0$.

\end{proof}
\begin{proof}[Proof of Corollary \ref{cor}]
As we are assuming $k \geq 3$, the inequality $p \leq g-k$ implies that $\deg(L-K_C) \leq g-1$, so we have $H^1(C,2K_C-L)=0$ for a general $L \in \text{Pic}^{2g+p-k+3}(C)$. To show that the condition ``$V_{g-p-3}^{g-p-4}(2K_C-L)$ has the expected dimension $g-k-p-1$'' holds, for $C$ a general $k$-gonal curve and $L$ general with $H^1(C,L)=H^1(C,2K_C-L)=0$, we need to show that
$$ 2K_C-L-C_{g-k-p} \nsubseteq K_C+C_{k-3}-C_{g-k+2},$$
see \cite{generic-secant}. This is equivalent to showing
$$ L-K_C+C_{g-k-p} \nsubseteq C_{g-k+2}-C_{k-3}.$$ For this, we may specialise $C$ to a hyperelliptic curve, as the $k$-gonality stratum in $\mathcal{M}_g$ contains the locus of hyperelliptic curves. In this case, the condition 
$$ L-K_C+C_{g-k-p} \subseteq C_{g-k+2}-C_{k-3}$$ implies
$$L-K_C \in C_{p+2}-C_{k-3}, $$
by \cite[Prop.\ 2.7]{generic-secant}. Under the assumption $p \leq g-k$, if $L$ is a general line bundle of degree $2g+p-k+3$, then $L-K_C$ does not lie in $C_{p+2}-C_{k-3}$. This completes the proof.
\end{proof}

\end{document}